\documentclass{amsart}
\usepackage{amsfonts}
\usepackage{esint}
\usepackage{color}
\usepackage{graphicx}

\usepackage{amsmath}

\usepackage{amssymb}

\usepackage{ulem}

\usepackage[mathscr]{eucal}

\newtheorem{thm}{Theorem}[section]
\newtheorem{lemma}[thm]{Lemma}

\newtheorem{corollary}[thm]{Corollary}
\theoremstyle{definition}

\newtheorem{example}[thm]{Example}
\theoremstyle{remark}

\numberwithin{equation}{section}

\numberwithin{equation}{section}
\newcommand{\R}{\mathbb R}

\newcommand{\ve}{\varepsilon}
\newcommand{\lam}{\lambda}

\def \R {\mathbb{R}}

\def \div {\mathrm{div}}

\def \dist {\mathrm{dist}}
\def \diam {\mathrm{diam}}

\def \suchthat {\ \big | \ }

\def \Leb {\mathscr{L}^n}

\def \ve {\varepsilon}

\newcommand{\defeq}{\mathrel{\mathop:}=}

\title[$\infty-$eigenvalues]{
Uniform stability of the ball with respect to the first Dirichlet and Neumann $\infty-$eigenvalues}

\begin{document}

\author[J.V. da Silva, J.D. Rossi and A.M. Salort]{Jo\~{a}o Vitor da Silva, Julio D. Rossi and Ariel M. Salort}

\address{Departamento de Matem\'atica, FCEyN - Universidad de Buenos Aires and
\hfill\break \indent IMAS - CONICET
\hfill\break \indent Ciudad Universitaria, Pabell\'on I (1428) Av. Cantilo s/n.
\hfill\break \indent Buenos Aires, Argentina.}

\email{jdasilva@dm.uba.ar, jrossi@dm.uba.ar, asalort@dm.uba.ar}
\urladdr{http://mate.dm.uba.ar/~jrossi, http://mate.dm.uba.ar/~asalort}



\begin{abstract}
In this note we analyze how perturbations of a ball $\mathfrak{B}_r \subset \R^n$ behaves in terms of their first (non-trivial) Neumann and Dirichlet $\infty-$eigenvalues when a volume constraint $\Leb(\Omega) = \Leb(\mathfrak{B}_r)$ is imposed.
Our main result states that $\Omega$ is uniformly close to a ball when it has first Neumann and Dirichlet eigenvalues close to the ones for the ball of the same volume $\mathfrak{B}_r$. In fact, we show that, if
$$
 |\lam_{1,\infty}^D(\Omega) - \lam_{1,\infty}^D(\mathfrak{B}_r)| = \delta_1 \quad \text{and} \quad |\lam_{1,\infty}^N(\Omega) - \lam_{1,\infty}^N(\mathfrak{B}_r)| = \delta_2,
$$
then there are two balls such that
$$\mathfrak{B}_{\frac{r}{\delta_1 r+1}} \subset \Omega \subset \mathfrak{B}_{\frac{r+\delta_2 r}{1-\delta_2 r}}.$$
In addition, we also obtain a result concerning stability of the Dirichlet $\infty-$eigen-functions.
\end{abstract}	
\maketitle


\section{Introduction}\label{Intro}

Let $\Omega \subset \R^n$ be a bounded domain (connected open subset) with smooth boundary, $1<p< \infty$ and $\Delta_p u \defeq \div(|\nabla u|^{p-2}\nabla u)$ (the standard $p$-Laplacian operator). Historically (cf. \cite{Lindq90}), it well-known that the first eigenvalue (referred as \textit{the principal frequency} in physical models) of the $p-$Laplacian  Dirichlet eigenvalue problem
\begin{equation}\label{eq.p}
\left\{
\begin{array}{rclcl}
  -\Delta_p u & = & \lambda_{1, p}^D(\Omega)|u|^{p-2} u & \text{in} & \Omega \\
  u & = & 0 & \text{on} & \partial \Omega
\end{array}
\right.
\end{equation}
can be characterized variationally as the minimizer of the following (normalized) problem:
\begin{equation}\tag{{\bf \text{p-Dirichlet}}}\label{1er.p}
   \displaystyle \lambda_{1, p}^D(\Omega) \defeq \inf_{u \in W^{1,p}_0 (\Omega) \setminus \{0\}}
   \left\{ \displaystyle \int_{\Omega} |\nabla u|^p dx  : \int_{\Omega} |u|^pdx = 1\right\}.
\end{equation}

In the theory of shape optimization and nonlinear eigenvalue problems obtaining (sharp) estimates for the eigenvalues in terms of geometric quantities of the domain (e.g. measure, perimeter, diameter, among others) plays a fundamental role due to several applications of these problems in pure and applied sciences. We recall that the explicit value to \eqref{1er.p} is known only for some specific values of $p$ or for very particular domains $\Omega$. Notice that upper bounds for $\lambda_{1, p}^D(\Omega)$ are usually obtained by selecting particular test functions in \eqref{1er.p}. Nevertheless, lower bounds are a more challenging task. In this direction we have the remarkable \textit{Faber-Krahn inequality}:
\textit{Among all domains of prescribed volume the ball minimizes \eqref{1er.p}}. More precisely,
\begin{equation} \label{chi}
   \lambda_{1, p}^D(\Omega) \geq \lambda_{1, p}^D(\mathfrak{B}),
\end{equation}
where $\mathfrak{B}$ is the $n$-dimensional ball such that $\Leb(\Omega) = \Leb(\mathfrak{B})$
(along this paper $\Leb (\Omega)$ will denote the Lebesgue measure of $\Omega$ that is assumed to be fixed).
Using isoperimetric or isodiametric inequality similar lower bounds for \eqref{1er.p} in terms of the perimeter (resp. diameter) of $\Omega$ are also available
(cf. \cite{Bhat} and \cite[page 224]{Lindq92}, and the references therein).
Recently, stability estimates for certain geometric inequalities were established in \cite{FMP2}, thereby providing an improved version of \eqref{chi} by adding a suitable remainder term, i.e., 
$$
   \lambda_{1, p}^D(\Omega) \geq \lambda_{1, p}^D(\mathfrak{B})\left(1+
\gamma_{p, n} (\mathcal{S}(\Omega))^{2+p}\right),
$$
where $\mathcal{S}(\Omega)$ is the so-called \textit{Fraenkel asymmetry}
of $\Omega$, which is precisely defined as
$$
   \mathcal{S}(\Omega) \defeq  \inf_{x_0 \in \R^n} \left\{ \frac{\Leb
(\Omega \bigtriangleup \mathfrak{B}_{r}(x_0))}{\Leb(\Omega)} \,
:\Leb(\mathfrak{B}_{r}(x_0)) = \Leb(\Omega) \right\},
$$
and $\gamma_{p, n}$ is a constant.
Observe that $\mathcal{S}$ measures the distance of a set $\Omega$
from being a ball. For such quantitative estimates
and further related topics we quote \cite{Bp}, \cite{Cianci}, \cite{Fusco}
and references therein.

Our main goal here is to find stability results for the limit case $p=\infty$.

First, we   introduce what is known for the limit as $p\to \infty$ in the eigenvalue problem for the
$p-$Laplacian.
When one takes the limit as $ p\to \infty$ in the minimization problem \eqref{1er.p}, one obtains
\begin{equation} \tag{{\bf \text{$\infty$-Dirichlet}}}\label{lam.infD}
  \lam_{1,\infty}^D(\Omega) \defeq \lim_{p\to\infty} \sqrt[p]{\lam_{1, p}^D(\Omega)} =\inf_{u \in W^{1, \infty}_0(\Omega)\setminus\{0\}} \|\nabla u\|_{L^{\infty}(\Omega)}>0,
\end{equation}
see \cite{JLM}.
Concerning the limit equation, also in \cite{JLM} it is proved that any family of normalized eigenfunctions $\{u_p\}_{p>1}$ to \eqref{1er.p} converges (up to a  subsequence) locally uniformly to $u_\infty \in W^{1,\infty}_0 (\Omega)$,
a minimizer for \ref{lam.infD} with $\|u_\infty\|_{L^{\infty}(\Omega)}=1$.
Moreover, the pair $(u_\infty, \lam_{1, \infty}^D(\Omega))$ is a nontrivial solution to
\begin{equation}\label{eq.infty.p}
\left\{
\begin{array}{rclcl}
  \min\Big\{-\Delta_\infty v_\infty, |\nabla v_\infty|-\lam_{1,\infty}^D(\Omega)v_\infty\Big\} & = & 0 & \text{in} & \Omega \\
  v_\infty & = & 0 & \text{on} & \partial \Omega.
\end{array}
\right.
\end{equation}
Solutions to \eqref{eq.infty.p} must be understood in the viscosity sense (cf. \cite{CIL} for a survey) and
$ \Delta_\infty u(x) \defeq \nabla u(x)^TD^2u(x) \cdot \nabla u(x) $
is the well-known \textit{$\infty-$Laplace  operator}.
In addition, also in \cite{JLM}, it is given an interesting and useful geometrical characterization for \eqref{lam.infD}:
\begin{equation} \label{lam1}
	\lam_{1,\infty}^D(\Omega) = \Big(\max_{x \in \Omega} \dist(x, \partial \Omega)\Big)^{-1}.
\end{equation}
Such an information means that the ``principal frequency'' for the $\infty$-eigenvalue problem can be detected from the geometry of the domain: it is precisely the reciprocal of radius $\mathfrak{r}_\Omega>0$ of the largest ball inscribed in $\Omega$. For more references concerning the first eigenvalue \eqref{eq.infty.p} we refer to \cite{KH}, \cite{NRSanAS} and \cite{Yu}.

Now, let us turn our attention to Neumann boundary conditions and consider the following eigenvalue problem:
\begin{equation}\label{eq.p.n}
\left\{
\begin{array}{rclcl}
  -\Delta_p u & = & \lambda_{1, p}^N(\Omega)|u|^{p-2} u & \text{in} & \Omega \\
  \displaystyle |\nabla u|^{p-2}\tfrac{\partial u}{\partial \nu} & = & 0 & \text{on} & \partial \Omega.
\end{array}
\right.
\end{equation}
As before, we stress that the first non-zero eigenvalue of \eqref{eq.p.n} can also be characterized variationally as the minimizer of the following normalized problem:
\begin{equation}\tag{{\bf \text{p-Neumann}}}\label{1er.p.n}
   \displaystyle \lambda_{1, p}^N(\Omega) \defeq \inf_{u \in W^{1, p}(\Omega)} \left\{ \displaystyle \int_{\Omega} |\nabla u|^pdx: \int_{\Omega}| u|^p dx  = 1\,\, \text{and}\,\,\int_\Omega |u|^{p-2}udx =0\right\}.
\end{equation}
The celebrated \textit{Payne-Weinberger inequality} provides a lower bound  (on any convex domain $\Omega \subset \R^n$) for the first (non-trivial) Neumann $p-$eigenvalue (cf. \cite{ENT} and \cite{Valt})
\begin{equation} \label{ggg}
  \lambda_{1, p}^N(\Omega) \geq (p-1)\left(\frac{2\pi}{p\, \diam(\Omega)\, \sin(\frac{\pi}{p})}\right)^{p}.
\end{equation}
For a  stability estimate for this problem with $p=2$ we refer to \cite{Bp}.

When $p\to \infty$, the minimization problem \eqref{1er.p.n} becomes
\begin{equation}\tag{{\bf \text{$\infty$-Neumann}}}\label{lam.inf}
  \displaystyle \lam_{1,\infty}^N(\Omega) \defeq \lim_{p\to\infty} \sqrt[p]{\lam_{1, p}^N(\Omega)} = \inf_{ u\in W^{1,\infty}(\Omega)  \atop{\max\limits_{\Omega} u = -\min\limits_{\Omega} u = 1}} \|\nabla u\|_{L^\infty(\Omega)},
\end{equation}
see \cite{EKNT} and \cite{RosSaint}. Concerning the limit equation, also in \cite{EKNT} and  \cite{RosSaint}, it is proved that any family of normalized eigenfunctions $\{u_p\}_{p>1}$ to \eqref{1er.p.n} converges (up to subsequence) locally uniformly to $u_\infty \in W^{1,\infty}_0 (\Omega)$ with $\|u_\infty\|_{L^{\infty}(\Omega)}=1$.
Moreover, the pair $(u_\infty, \lam_{1, \infty}^N(\Omega))$ is a nontrivial solution to
\begin{equation}\label{eq.infty.p.n}
\left\{
\begin{array}{rclcl}
  \min\Big\{-\Delta_\infty v_\infty, |\nabla v_\infty|-\lam_{1,\infty}^N(\Omega)v_\infty\Big\} & = & 0 & \text{in} & \Omega\cap \{v>0\} \\
\max\Big\{-\Delta_\infty v_\infty, -|\nabla v_\infty|-\lam_{1,\infty}^N(\Omega)v_\infty\Big\} & = & 0 & \text{in} & \Omega\cap \{v<0\} \\
 -\Delta_\infty v_\infty & = & 0 & \text{in} & \Omega \cap \{v=0\}\\
 \displaystyle \frac{\partial v_{\infty}}{\partial \nu}& = & 0 & \text{in} & \partial \Omega.
\end{array}
\right.
\end{equation}
In addiction, we have the following geometrical characterization for $\lam_{1,\infty}^N(\Omega)$:
\begin{equation} \label{lam1.n}
	\lam_{1,\infty}^N(\Omega) = \frac{2}{\diam(\Omega)},
\end{equation}
where the intrinsic diameter of $\Omega$ is defined as
$$
	\diam(\Omega) \defeq \max_{\bar\Omega \times \bar\Omega} d_{\Omega}(x,y) = \max_{\partial \Omega \times \partial \Omega} d_{\Omega}(x,y),
$$
being $d_{\Omega}(x,y)$ the geodesic distance given by $ d_{\Omega}(x,y)=\inf_{\gamma} \text{Long}(\gamma)
$,  where the infimum is taken
over all possible Lipschitz curves in $\bar\Omega$ connecting $x$ and $y$.

We remark that in the limit case $p=\infty$, the geometrical characterization \eqref{lam1.n} of \eqref{lam.inf} yields several interesting consequences:
\begin{itemize}
  \item[\checkmark] If $\Leb(\Omega) = \Leb(\mathfrak{B})$, $\mathfrak{B}$ being a ball, then $\lam_{1,\infty}^N(\Omega) \leq \lam_{1,\infty}^N(\mathfrak{B})$, which establishes a \textit{Szeg\"{o}-Weinberger type inequality}: among all domains of prescribed volume the ball maximizes \eqref{lam.inf}.
  \item[\checkmark] $\lam_{1,\infty}^N(\Omega) \leq \lam_{1,\infty}^D(\Omega)$ for any convex $\Omega$ with equality if and only if $\Omega$ is a ball.
  \item[\checkmark] The Payne-Weinberger inequality, \eqref{ggg}, becomes an equality when  $p = \infty$.
\end{itemize}

Taking account the previous historic overview, we arrive to our main result, which establishes the stability of the ball with respect to small perturbations of their first Dirichlet and Neumann $\infty-$eigenvalues. More precisely, if a domain $\Omega\subset \R^n$ has Dirichlet and Neumann $\infty-$eigenvalues close enough to those of the ball
$\mathfrak{B}_r$ of the same Lebesgue measure, then $\Omega$ is uniformly ``almost'' ball-shaped.

\begin{thm}\label{Mainthm} Let $\Omega$ be an open domain satisfying $\Leb(\Omega)=\Leb(\mathfrak{B}_r)$. If for some $\delta_i>0$ ($i=1, 2$) small enough it holds that
$$
 |\lam_{1,\infty}^D(\Omega) - \lam_{1,\infty}^D(\mathfrak{B}_r)| = \delta_1 \quad \text{and} \quad |\lam_{1,\infty}^N(\Omega) - \lam_{1,\infty}^N(\mathfrak{B}_r)| = \delta_2,
$$
then there are two balls such that
$$\mathfrak{B}_{\frac{r}{\delta_1 r+1}} \subset \Omega \subset \mathfrak{B}_{\frac{r+\delta_2 r}{1-\delta_2 r}}.$$ \end{thm}

The previous theorem implies the following convergence result.

\begin{thm}\label{Mainthm2} Let $\{\Omega_k\}_{k \in \mathbb{N}}$ be a family of uniformly bounded domains satisfying $\Leb(\Omega_k)=\Leb(\mathfrak{B}_r)$. If
$$
   |\lam_{1,\infty}^D(\Omega_k) - \lam_{1,\infty}^D(\mathfrak{B}_r)| = \text{o}(1) \quad \text{and} \quad |\lam_{1,\infty}^N(\Omega) - \lam_{1,\infty}^N(\mathfrak{B}_r)| = \text{o}(1) \quad \text{as } k \to \infty,
$$
then $$\Omega_k \to \mathfrak{B}_r$$ in the sense that the Hausdorff distance between $\Omega$ and a ball
$\mathfrak{B}_r$ goes to zero, i.e.,
$$
d_\mathcal{H} (\Omega_k,  \mathfrak{B}_r) := \max\Big\{\,\sup _{{x\in \Omega_k}}\inf _{{y\in \mathfrak{B}_r}}d(x,y),\,\sup _{{y\in \mathfrak{B}_r}}\inf _{{x\in  \Omega_k}}d(x,y)\, \Big\} \to 0.
$$
\end{thm}

Note that our results imply that
\begin{equation} \label{ecccc}
  \max\left\{\Leb\left(\Omega\bigtriangleup \mathfrak{B}_{\frac{r}{\delta_1 r+1}}\right), \Leb\left(\Omega\bigtriangleup \mathfrak{B}_{\frac{r+\delta_2 r}{1-\delta_2 r}}\right)\right\}\leq \mathfrak{C}(n, \delta_i, r)r^n.
\end{equation}
where $\mathfrak{C}(n, \delta_i, r)=\omega_n \max\{(\delta_1r+1)^n-1,(n-1)\delta_2\}\to 0$ as $\delta_i\to 0$.
Hence, we can control the Fraenkel asymmetry of the set, $S(\Omega)$. But our results give much more since we have a sort
of uniform control on how far the set is from being a ball (for instance, we have convergence in Hausdorff distance
in Theorem \ref{Mainthm2}).

Another important question in this theory consists on how the corresponding $\infty-$ground states
(solutions to \eqref{eq.infty.p}) behave in relation to perturbations of the $\infty-$eigenvalues of the ball. The next result provides an answer for this issue, showing that Dirichlet $\infty-$eigenfunctions are uniformly close to a cone when the first Dirichlet and Neumann $\infty-$eigenvalues are close to those for the ball.
Note that, in general, the $\infty-$eigenvalue problem \eqref{eq.infty.p} may have multiple solutions (the first eigenvalue may not be simple), see \cite{HSY} and \cite{Yu}.

\begin{thm} \label{teo.autofunc.intro}
Let $\Omega$ be an open domain satisfying $\Leb(\Omega)=\Leb(\mathfrak{B}_r)$. Given $\varepsilon >0$
there are $\delta_i(\varepsilon)>0$ ($i=1, 2$) small enough such that: if
$$
 |\lam_{1,\infty}^D(\Omega) - \lam_{1,\infty}^D(\mathfrak{B}_r)| < \delta_1 \quad \text{and} \quad |\lam_{1,\infty}^N(\Omega) - \lam_{1,\infty}^N(\mathfrak{B}_r)| < \delta_2,
$$
then
$$
  |u(x)-v_{\infty}(x)| < \varepsilon \,\,\, \text{in} \,\,\, \Omega \cap \mathfrak{B}_r,
$$
where $$v_{\infty} (x)= 1 - \frac{|x|}{r}$$ is the normalized $\infty-$ground state to \eqref{eq.infty.p} in $\mathfrak{B}_r$.
\end{thm}

Theorem \ref{teo.autofunc.intro} can be rewritten as follows:

\begin{corollary}\label{CorConv} Let $\{u_k\}_{k \in \mathbb{N}}$ be a family of normalized solutions to \eqref{eq.infty.p} in $\Omega_k$ such that 
$$
 |\lam_{1,\infty}^D(\Omega_k) - \lam_{1,\infty}^D(\mathfrak{B}_r)| = \text{o}(1) \quad \text{and} \quad 
 |\lam_{1,\infty}^N(\Omega_k) - \lam_{1,\infty}^N(\mathfrak{B}_r)| = \text{o}(1)\quad \text{as } k \to \infty.
$$
Then,
$$
  u_k \to v_{\infty} \quad \text{locally uniformly} \quad \text{in} \quad \mathfrak{B}_r,
$$
where $$v_{\infty} (x)= 1 - \frac{|x|}{r}$$ is the normalized $\infty-$ground state to \eqref{eq.infty.p} in $\mathfrak{B}_r$.
\end{corollary}

Our approach can be applied for other classes of operators with $p-$Laplacian type structure.
We can deal with $p-$Laplacian type problems involving an \textit{anisotropic $p$-Laplacian operator}
$$
   \displaystyle -\mathcal{Q}_p u \defeq -\div(\mathbb{F}^{p-1}(\nabla u)\mathbb{F}_{\xi}(\nabla u)),
$$
where $\mathbb{F}$ is an appropriate (smooth) norm of $\R^n$ and $1<p< \infty$. The necessary tools for studying the anisotropic Dirichlet eigenvalue problem, as well as its limit as $p \to \infty$ can be found in \cite{BKJ}.
Here, to obtain results similar to ours, one has to replace Euclidean balls with  balls in the norm $\mathbb{F}$.

\medskip

The paper is organized as follows: in Section \ref{sect-main} we prove our main stability results
including the behavior of the corresponding $\infty-$eigenfunctions
and in Section \ref{sect-examples} we collect several
examples that illustrate our results.

\section{Proof of the Main Theorems} \label{sect-main}

Before proving our main result we introduce some notations which will be used throughout this section. Given a bounded domain $\Omega\subset \R^n$ and a ball $\mathfrak{B}_r\subset\R^n$ of radius $r>0$
we denote $\lam_{1,\infty}^D(\Omega)$ and $\lam_{1,\infty}^D(\mathfrak{B}_r)$ the first Dirichlet eigenvalues \eqref{lam1} in $\Omega$ and  in $\mathfrak{B}_r$, respectively; analogously, $\lam_{1,\infty}^N(\Omega)$ and $\lam_{1,\infty}^N(\mathfrak{B}_r)$ stand for the first nontrivial Neumann eigenvalues \eqref{lam1.n} in $\Omega$ and in  $\mathfrak{B}_r$.

We introduce the following class of sets which will play an important role in our approach. For non-negative constants $\delta_1$ and $\delta_2$ we define the class:
$$
  \Xi_{\delta_1, \delta_2}(\mathfrak{B}_r) \defeq \left\{
\begin{array}{c}
  \Omega \subset \R^n \\
  \text{bounded domain with} \\
   \Leb(\Omega)  =  \Leb(\mathfrak{B}_r)
\end{array} :
\begin{array}{rcl}
  |\lam_{1,\infty}^D(\Omega)-\lam_{1,\infty}^D(\mathfrak{B}_r)| & = & \delta_1 \\[10pt]
  |\lam_{1,\infty}^N(\Omega)-\lam_{1,\infty}^N(\mathfrak{B}_r)| & = & \delta_2
\end{array}
\right\}.
$$
 Notice that, $\Xi_{0, 0}(\mathfrak{B}_r)$ consists of the family of all balls with radius $r>0$. Another important remark is that the elements of $\Xi_{\delta_1, \delta_2}(\mathfrak{B}_r)$ are invariant by rigid movements (rotations, translations, etc).

Similarly, we can define the class $\Xi^D_{\delta_1}(\mathfrak{B}_r)$ (resp. $\Xi^N_{\delta_2}(\mathfrak{B}_r)$) as being $\Xi_{\delta_1, \delta_2}(\mathfrak{B}_r)$ with the restriction on the Dirichlet (resp. Neumann) eigenvalues only.

In the next lemma we show that a control on the difference of the first Dirichlet eigenvalue implies that $\Omega$ contains a large ball.

\begin{lemma}\label{Lemma1}
If $\Omega\in \Xi^D_{\delta_1}(\mathfrak{B}_r)$ then there exists a ball such that $$\mathfrak{B}_{\frac{r}{\delta_1 r+1}} \subset \Omega.$$
 Moreover,
$$
   \Leb\left(\Omega\bigtriangleup \mathfrak{B}_{\frac{r}{\delta_1 r +1}}\right)\leq \mathfrak{c}(n, \delta_1, r) r^{n}.
$$
where $\mathfrak{c} = \text{o}(1)$ as $\delta_1 \to 0$.
\end{lemma}

\begin{proof}
According to \eqref{lam1} we have that
$$
   \delta_1=|\lam_{1,\infty}^D(\Omega) - \lam_{1,\infty}^D(\mathfrak{B}_r)|=\left| \frac{1}{r_\Omega} - \frac{1}{r}\right|.
$$
It follows that
$$
	r_\Omega \geq \frac{r}{\delta_1 r +1} .
$$
and then there is ball such that $$\mathfrak{B}_{\frac{r}{\delta r+1}} \subset \Omega. $$
Finally,
\begin{align*}
\Leb(\Omega\triangle \mathfrak{B}_{\frac{r}{\delta r+1}} ) & = \Leb(\Omega) - \Leb(\mathfrak{B}_{\frac{r}{\delta r+1}} )
\\
 & = \omega_n r^n \left(1-\frac{1}{(\delta r +1)^n}\right)\\
& \leq  \omega_n r^n \left((\delta r +1)^n-1\right)\\
&= \mathfrak{c}(n, \delta, r)r^{n}
\end{align*}
and the lemma follows.
\end{proof}

Now, we show that a control on the difference of the first Neumann eigenvalue
implies that $\Omega$ is contained in a small ball.

\begin{lemma}\label{Lemma2}
If $\Omega\in \Xi^N_{\delta_2}(\mathfrak{B}_r)$ then there is a ball such that
 $$\Omega \subset \mathfrak{B}_{\frac{r}{1-\delta_2 r}}.$$
 Moreover,
$$
  \Leb\left(\Omega \bigtriangleup B_{\frac{r}{1-\delta_2 r}}\right)\leq (n-1)\omega_n r^n \delta_2.
$$
\end{lemma}

\begin{proof}
Using \eqref{lam1.n} we have that
$$
   \delta_2=|\lam_{1,\infty}^N(\Omega) - \lam_{1,\infty}^N(\mathfrak{B}_r)|=\left| \frac{2}{\diam(\Omega)} - \frac{1}{r}\right|.
$$
It follows that
$$
	\text{diam}(\Omega)\leq \frac{2r}{1-\delta_2 r } = r+\frac{r(1+\delta r)}{1-\delta_2 r}
$$
and then there exists a ball such that
$$
  \Omega \subset \mathfrak{B}_{\frac{\diam(\Omega)}{2}} = \mathfrak{B}_{\frac{r}{1-\delta_2 r}}.$$
Moreover,
\begin{align*}
\Leb\left(\Omega\bigtriangleup \mathfrak{B}_{\frac{\diam(\Omega)}{2}}\right)&=\Leb\left(\mathfrak{B}_{\frac{\diam(\Omega)}{2}}\right) - \Leb(\Omega)
\\
& =\omega_n r^n \left( \left( 1+\frac{\delta_2}{1-\delta_2 r}\right)^n -1 \right)\\
&=\omega_n r^n \delta_2 \sum_{k=2}^n \left(\frac{\delta_2}{1-\delta_2 r}\right)^k\\
&\leq (n-1)\omega_n \delta_2 r^n
\end{align*}
and the lemma follows.
\end{proof}

\begin{proof}[Proof of Theorem \ref{Mainthm}] The proof of Theorem \ref{Mainthm} follows as an immediate consequence of Lemmas \ref{Lemma1} and \ref{Lemma2}.
\end{proof}

Next, we will prove Theorem \ref{Mainthm2}.

\begin{proof}[Proof of Theorem \ref{Mainthm2}] The hypothesis implies that $\Omega_k \in \Xi_{\delta_k, \varepsilon_k}(\mathfrak{B}_r)$ for $\delta_k, \varepsilon_k =\text{o}(1)$ as $k \to \infty$. For this  reason, by Theorem \ref{Mainthm} there are two balls such that
$$
  \mathfrak{B}_{\frac{r}{\delta_k r+1}} \subset \Omega_k \subset \mathfrak{B}_{\frac{r+\varepsilon_k r}{1-\varepsilon_k r}}.
$$
Now, using that all these balls are centered at points that are bounded (since we assumed that the family $\Omega_k $
is uniformly bounded), we can extract a subsequence such that the centers converge and therefore we conclude that
there is a ball $\mathfrak{B}_r$ such that
$\Omega_k \to \mathfrak{B}_r$ as $k \to \infty$.
\end{proof}

\begin{proof}[Proof of Theorem \ref{teo.autofunc.intro}] 
The proof follows by contradiction. Let us suppose that there exists an $\varepsilon_0>0$ such that the thesis of Theorem fails to hold. This means that for each $k \in \mathbb{N}$ we might find a domain $\Omega_k$ and $u_k$, a normalized $\infty-$ground state to \eqref{eq.infty.p} in $\Omega_k$, such  that $\Omega_{k} \in \Xi_{\gamma_k, \zeta_k}(\mathfrak{B}_r)$ with $\gamma_k, \zeta_k = \text{o}(1)$ as $k \to \infty$, that is, 
$$
 |\lam_{1,\infty}^D(\Omega_k) - \lam_{1,\infty}^D(\mathfrak{B}_r)| < \gamma_k \quad \text{and} \quad 
 |\lam_{1,\infty}^N(\Omega_k) - \lam_{1,\infty}^N(\mathfrak{B}_r)| < \zeta_k,
$$
with $\gamma_k, \zeta_k = \text{o}(1)$ as $k \to \infty$,
together with 
\begin{equation}\label{Eqcont}
     |u_k(x) - v_{\infty}(x)|> \varepsilon_0 \quad \,\,\text{in} \,\,\, \Omega_k \cap \mathfrak{B}_r,
\end{equation}
for every $k \in \mathbb{N}$.

Using our previous results, we can suppose that every $\Omega_k \subset \mathfrak{B}_{2r}$. Then, by extending $u_k$ to zero outside of $\Omega_k$, we may assume that $\{u_k\}_{k\in \mathbb{N}} \subset W_0^{1, \infty}(\mathfrak{B}_{2r})$. In this context,
standard arguments using viscosity theory show that, up to a subsequence, $u_k \to u_{\infty}$ uniformly in
$\overline{\mathfrak{B}_{2r}}$, being the limit $u_\infty$ a normalized eigenfunction for some domain $\hat{\Omega}$ with $\hat{\Omega} \Subset\mathfrak{B}_{2r}$.
Moreover, we have that $\lambda^D_{1,\infty} (\Omega_k) \to \lambda^D_{1,\infty} (\hat{\Omega})$.

According to Theorem \ref{Mainthm2}, $\Omega_k \to \mathfrak{B}_r$ as $k \to \infty$. By the previous sentences we conclude that $\hat{\Omega} = \mathfrak{B}_r$. Now, by uniqueness of solutions to \eqref{eq.infty.p} in $\mathfrak{B}_r$ we conclude that $u_{\infty} = v_{\infty}$. However, this contradicts \eqref{Eqcont} for $k \gg 1$ (large enough). Such a contradiction proves the theorem.
\end{proof}

\section{Examples}\label{sect-examples}

Given a fixed ball $\mathfrak{B}$ and a domain $\Omega$ having  both of them the same volume, Theorem \ref{Mainthm} says that if the $\infty-$eigenvalues are close each other then $\Omega$ is almost ball-shaped uniformly.
The following examples illustrate Theorem \ref{Mainthm} and \ref{Mainthm2}.

\begin{example}
The reciprocal in Theorem \ref{Mainthm} (and Theorem  \ref{Mainthm2}) is not true: given a fixed ball $\mathfrak{B}$, clearly, there are domains $\Omega$ fulfilling \eqref{ecccc} such that the difference between the Neumann (and Dirichlet) eigenvalues in $\Omega$ and in $\mathfrak{B}$ is not small. Let us present some illustrative examples.

\begin{enumerate}
  \item A stadium. Let  $\mathfrak{B}$ be the unit ball in $\R^2$ and $\Omega$ the stadium domain given in Figure \ref{fi1} (a) with $\ell=\frac{\pi(1-\ve^2)}{2\ve}$. In this case $\Leb(\mathfrak{B})=\Leb(\Omega)=\pi$ for any $0<\ve<1$. However,
$$
\lam_{1,\infty}^N(\mathfrak{B})=1, \qquad \lam_{1,\infty}^N(\Omega)=\frac{2}{\text{diam}(\Omega)}= \frac{4\ve}{\pi+\ve^2(4-\pi)}<\frac13 \quad \text{if } \ve<\frac14.
$$

\item A ball with holes. If $\Omega=B(0,\sqrt{1+\ve^2}) \setminus B(0,\ve)$ is the domain given in Figure \ref{fi1} (b), then $\Leb(\mathfrak{B})=\Leb(\Omega)=\pi$, however
$$
\lam_{1,\infty}^D(\mathfrak{B})=1, \qquad \lam_{1,\infty}^D(\Omega)=\frac{1}{\sqrt{1+\ve^2}}>\frac32 \quad \text{if } \frac34<\ve<1.
$$

\item A ball with thin tubular branches. If  $\Omega$ is the domain given in Figure \ref{fi1} (c), the condition $\Leb(\mathfrak{B})= \Leb(\Omega)$ gives the relation
$$
	r(r+\ve) + \ve(\tfrac{1}{\pi}+\tfrac{\ve}{2})=1, \qquad \text{diam}(\Omega)=1+r+\pi(1+r).
$$
For instance, if we take $\ve=10^{-3}$ it follows that $r\sim 0.999465$ and then
$$
\lam_{1,\infty}^N(\mathfrak{B})=\frac{2}{\text{diam}(\mathfrak{B})}=1, \qquad \lam_{1,\infty}^N(\Omega)=\frac{2}{\text{diam}(\Omega)}\sim 0.2415.
$$
\begin{figure}[ht]\label{fi1}
\begin{center}
\includegraphics[width=12cm,height=4.7cm]{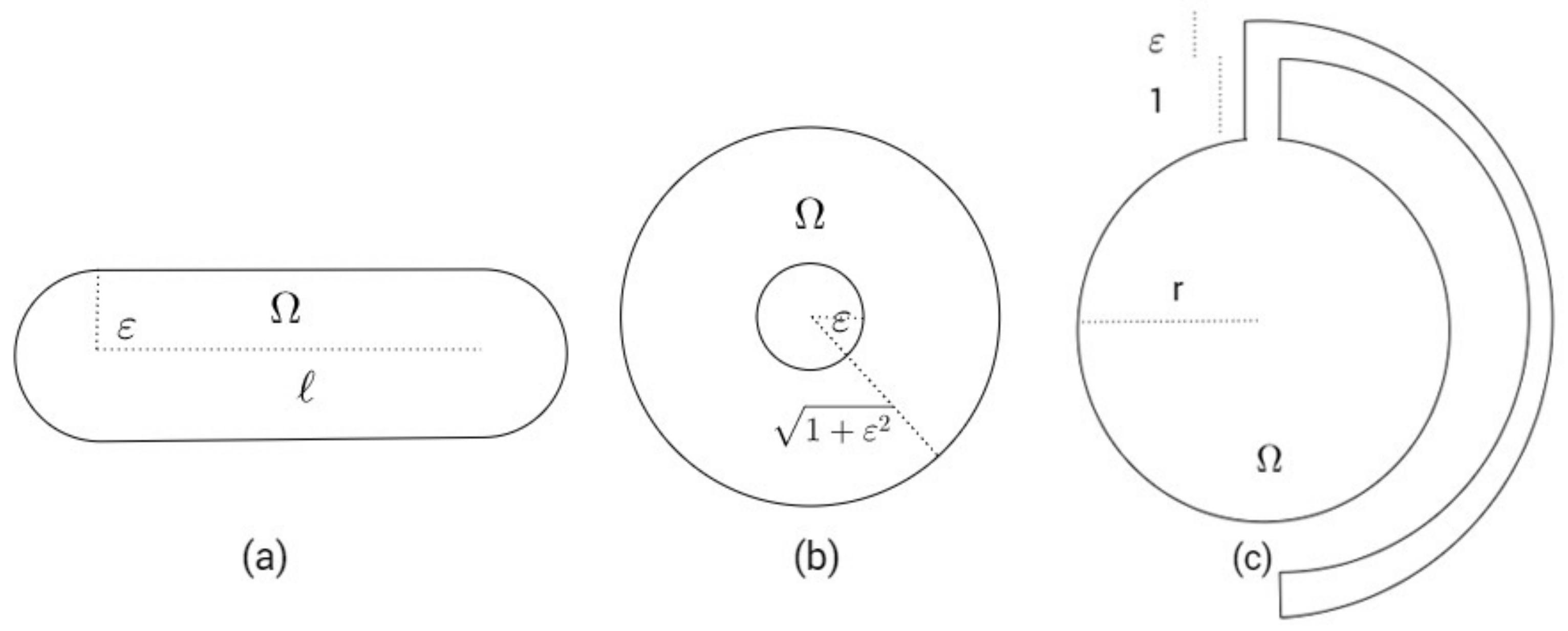}
\caption{Three examples of domains}
\label{dib3}
\end{center}
\end{figure}
\end{enumerate}

Hence, in view of these examples we conclude that a domain that has Dirichlet and Neumann $\infty-$eigenvalues
close to the ones for the ball is close to a ball not only in the sense that
$\Leb\left(\Omega\bigtriangleup \mathfrak{B}_{r}\right)$ is small but it
can not contain holes deep inside (small holes near the boundary are allowed) and can not
have thin tubular branches.
\end{example}

\begin{example} The regular polygon $\mathbb{P}_k$ of $k-$sides ($k\geq 3$) centered at the origin such that $\Leb(\mathbb{P}_k) = \Leb(\mathfrak{B}_r)$ satisfies
$$
  \displaystyle | \lam_{1,\infty}^D(\mathbb{P}_k) - \lam_{1,\infty}^D(B_r)| = \delta_1\quad \text{and} \quad |\lam_{1,\infty}^N(B_r) - \lam_{1,\infty}^N(\mathbb{P}_k)| = \delta_2,
$$
where
$$
	\delta_1=\frac{1}{r\sqrt{\frac{\pi}{k\tan(\frac{\pi}{k})}}}- \frac{1}{r} \quad \text{ and} \quad \delta_2 = \frac{1}{r} - \frac{1}{r\sqrt{\frac{2\pi}{k\sin(\frac{2\pi}{k})}}}.
$$
Therefore, we can recover the well known convergence $\mathbb{P}_k \to \mathfrak{B}_r$ as $k \to \infty$.
\end{example}

\begin{example} Given $k \in \mathbb{N}$ and positive constants $\mathfrak{a}^k_1,\cdots, \mathfrak{a}^k_n$, the $n-$dimensional ellipsoid given by $$\displaystyle \mathcal{E}_k \defeq \left\{(x_1, \cdots, x_n) \suchthat \sum_{i=1}^{n} \Big(\frac{x_i}{\mathfrak{a}^k_i}\Big)^2<1\right\}$$ such that $\Leb(\mathcal{E}_k) = \Leb(\mathfrak{B}_r)$ satisfies
$$
  \displaystyle | \lam_{1,\infty}^D(\mathcal{E}_k) - \lam_{1,\infty}^D(B_r)| =\delta_1 \quad \text{and} \quad |\lam_{1,\infty}^N(B_r) - \lam_{1,\infty}^N(\mathcal{E}_k)| = \delta_2,
$$
where
$$
	\delta_1  = \frac{1}{\min\limits_{i}\{\mathfrak{a}^k_i\}} - \frac{1}{r}, \quad \text{and } \quad \delta_2=\frac{1}{r} - \frac{1}{\max\limits_{i}\{\mathfrak{a}^k_i\}}.
$$
Therefore, we recover the fact that if $\min\limits_{i} \mathfrak{a}^k_i \to r$ and $\max\limits_{i} \mathfrak{a}^k_i \to r$ as $k \to \infty$, then $\mathcal{E}_k \to \mathfrak{B}_r$.
\end{example}

\begin{example} Given $r>0$ let $k_0 \in \mathbb{N}$ such that $\frac{1}{2\pi} \sqrt{\frac{4}{k^2} + 4 \pi^2r^2}> \frac{1}{k \pi}$ for all $k \geq k_0$. For each $k\in \mathbb{N}$ let $\Omega_k$ be the planar stadium domain from Figure 1 (a) with $l_k = \frac{1}{k}$ and $\varepsilon_k = \frac{1}{2\pi} \sqrt{\frac{4}{k^2} + 4 \pi^2r^2}-\frac{1}{k \pi}$. It is easy to check that $\Omega_k \in \Xi_{\frac{1}{\varepsilon_k}-\frac{1}{r}, \frac{2}{2\varepsilon_k + \frac{1}{k}}-\frac{1}{r}}(\mathfrak{B}_r)$.  Furthermore, in  this case we have that the eigenfunctions are explicit and given by
$$
   u_k(x) = \frac{1}{\varepsilon_k}\dist(x, \partial \Omega_k).
$$
Finally, form Corollary \ref{CorConv}
$$
  u_k(x) \to v_{\infty}(x) = \frac{1}{r}\dist(x, \partial \mathfrak{B}_r) \quad \text{locally uniformly} \quad \text{in} \quad \mathfrak{B}_r \quad \text{as}\quad k \to \infty.
$$
\end{example}

\subsubsection*{Acknowledgments}
This work was supported by Consejo Nacional de Investigaciones Cien-t\'{i}ficas y T\'{e}cnicas (CONICET-Argentina). JVS would like to thank the Dept. of Math. and FCEyN Universidad de Buenos Aires for providing an excellent working environment and scientific atmosphere during his Postdoctoral program.


\begin{thebibliography}{99}

\bibitem{Bhat}
T. Bhattacharia, \textit{A proof of the Faber-Krahn inequality for the first eigenvalue
of the $p-$Laplacian}, Ann. Mat. Pura Appl. Ser. 4 177 (1999) 225--240.


\bibitem{Bp} L. Brasco and A. Pratelli, {\it Sharp stability of some spectral inequalities}. Geom. Funct. Anal. 22 (2012), no. 1, 107--135.

  \bibitem{BKJ}
M. Belloni, B. Kawohl and P. Juutinen. {\it The $p$-Laplace eigenvalue problem as $p \to \infty$ in a Finsler metric}. J. Eur. Math. Soc. 8 (2006), 123--138.

\bibitem{Cianci} A. Cianchi, N. Fusco, F. Maggi and A. Pratelli, {\it The sharp Sobolev inequality in quantitative form}.
J. Eur. Math. Soc. (JEMS) 11 (2009), no. 5, 1105--1139.


\bibitem{HSY}
R. Hynd, C. K. Smart and Y. Yu, {\it Nonuniqueness of infinity ground states}. Calc. Var. Partial Differential Equations (2013) 48, 545--554.\label{HSY}


\bibitem{CIL}
M.~G. Crandall, H. Ishii, and P.-L. Lions, \textit{User's guide
  to viscosity solutions of second order partial differential equations}, Bull.
  Amer. Math. Soc. (N.S.) {27} (1992), no.~1, 1--67.

%


\bibitem{EKNT}
L. Esposito, B. Kawohl, C. Nitsch and C. Trombetti, \textit{The Neumann eigenvalue
problem for the $\infty$-Laplacian}, Rend. Lincei Mat. Appl. 26 (2015), 119--134.\label{EKNT}

\bibitem{ENT}
L. Esposito, C. Nitsch and C. Trombetti, \textit{Best constants in Poincar\'{e} inequalities for convex
domains}, J. Conv. Anal. 20 (2013) 253--264.

\bibitem{Fusco} N. Fusco. {\it The quantitative isoperimetric inequality and related topics}. Bull. Math. Sci. (2015) 5, 517--607

\bibitem{FMP2} N. Fusco, F. Maggi, A. Pratelli, {\it Stability estimates for certain Faber--Krahn,
isocapacitary and Cheeger inequalities}, Ann. Sc. Norm. Super. Pisa Cl. Sci. 8 (2009),
51--71.

\bibitem{JLM}
P. Juutinen, P. Lindqvist and J. Manfredi \textit{The $\infty$-eigenvalue problem}.
Arch. Ration. Mech. Anal. 148 (1999), no. 2, 89--105.\label{JLM}

\bibitem{KH}
B. Kawohl and J. Horak, \textit{On the geometry of the $p-$Laplacian operator}.
Discr. Cont. Dynamical Syst., 10, no 4, 2017,  799--813.


\bibitem{Lindq90}
P. Lindqvist, \textit{On the equation $\div(|\nabla u|^{p-2}\nabla u) + \lambda |u|^{p-2}u$}, Proc. AMS, 109
(1990), 157-164.\label{Lindq90}

\bibitem{Lindq92}
P. Lindqvist, \textit{A note on the nonlinear Rayleigh quotient}, in: Analysis, Algebra
and computers in mathematical research (Lulea 1992), Eds.: M.Gyllenberg $\&$
L.E.Persson Marcel Dekker Lecture Notes in Pure and Appl. Math. 156 (1994) 223--231

\bibitem{NRSanAS}
J. C. Navarro, J.D. Rossi, A. San Antolin, and N. Saintier, \textit{The dependence of the first eigenvalue of the infinity Laplacian with respect to the domain}, Glasg. Math. J. {56} (2014), no. 2, 241--249.

\bibitem{RosSaint}
J. D. Rossi and N.Saintier, \textit{On the first nontrivial eigenvalue of the $\infty$-Laplacian with Neumann boundary conditions}. Houston Journal of Mathematics.  42(2), 613--635, (2016).\label{RosSaint}

\bibitem{Valt}
D. Valtorta, \textit{Sharp estimate on the first eigenvalue of the $p-$Laplacian}, Nonlinear Anal.
75 (2012), 4974--4994.

\bibitem{Yu}
Y. Yu, \textit{Some properties of the ground states of the infinity Laplacian}, Indiana Univ. Math. J. {56} (2007), 947--964.


\end{thebibliography}
\end{document}